\documentclass[
  a4paper, 
  reqno, 
  oneside, 
  11pt
]{amsart}

\usepackage[utf8]{inputenc}

\usepackage[dvipsnames]{xcolor} % Must come before tikz!
\colorlet{cite}{LimeGreen!50!Green}
\usepackage{tikz}  % for diagrams
\usetikzlibrary{arrows,positioning}
\tikzset{ 
  baseline=-2.3pt,
  text height=1.5ex, text depth=0.25ex,
  >=stealth,
  node distance=2cm,
  mid/.style={fill=white,inner sep=2.5pt},
}

\usepackage{amsthm, amssymb, amsfonts}	% Standard maths packages.

\usepackage{graphicx,caption,subcaption}

\usepackage{braket}  % For nice sets
\usepackage{fouriernc}  % Nice maths font
\usepackage{microtype}

\usepackage[%
  bookmarks=true,			% show bookmarks bar?
  unicode=true,			% non-Latin characters in Acrobat’s bookmarks
  pdftitle={Master's Thesis},		%
  pdfauthor={Brian Callander},	% 
  pdfkeywords={} {},	% list of keywords
  colorlinks=true,		% false: boxed links; true: colored links
  linkcolor=Red,			% color of internal links
  citecolor=cite,		% color of links to bibliography
  filecolor=magenta,		% color of file links
]{hyperref}				% One of the most useful packages I have found.
\usepackage{cleveref}

\newtheoremstyle{mydef}
  {}		% Space above environment
  {}		% Space below environment
  {}		% Body font
  {}		% Indent amount (empty = no indent, \parindent = para indent)
  {\scshape}	% theorem head font
  {. }		% Punctuation after heading
  { }		% Space after heading
  {\thmname{#1}\thmnumber{ #2}\thmnote{ #3}}	% Heading spec

\theoremstyle{plain}	% 'plain' is the default.  The others are 'definition' and 'remark'.
\newtheorem{theorem}{Theorem} % putting [section] on the end here tells latex to number the theorem environment within sections, ie. Theorem 2.3 for the third theorem in section 2.
\newtheorem{corollary}[theorem]{Corollary}

\theoremstyle{mydef} % Here we have used the custom theoremstyle defined above instead of the usual 'definition' style.

\theoremstyle{remark}

\newtheorem{remark}{Remark}

\crefname{part}{Part}{Parts}
\crefname{chapter}{Chapter}{Chapters}
\crefname{section}{Section}{Sections}
\crefname{theorem}{Theorem}{Theorems}
\crefname{proposition}{Proposition}{Propositions}
\crefname{lemma}{Lemma}{Lemmata}
\crefname{corollary}{Corollary}{Corollaries}
\crefname{definition}{Definition}{Definitions}
\crefname{example}{Example}{Examples}
\crefname{remark}{Remark}{Remarks}
\crefname{notation}{Notation}{Notations}
\crefname{figure}{Figure}{Figures}
\crefname{enumi}{Item}{Items}

\newcommand{\inner}[1]{\left\langle #1 \right\rangle}
\newcommand{\abs}[1]{\left\lvert #1 \right\rvert}

\DeclareMathOperator{\diag}{Diag}

\DeclareMathOperator{\id}{id}

\author{Brian Callander and Elizabeth Gasparim}
\date{\today}
\title{Hodge Diamonds of Adjoint Orbits}
\thanks{We are grateful to Koushik Ray and Pushan Majumdar of the Department of Theoretical
  Physics, Indian Association For The Cultivation of Science, Kolkata, for
  running  our large memory (34GB)  computations on their servers. We would
  also like to thank Daniel Grayson for the time he has generously spent
  assisting us with technical issues.
}
\begin{document}
\maketitle

\begin{abstract}
We present a Macaulay2 package that produces compactifica-
tions of adjoint orbits and of the fibres of symplectic Lefschetz fibrations on
them. We use Macaulay2 functions to calculate the corresponding Hodge
diamonds, which then reveal topological information about the Lefschetz
fibrations.
\end{abstract}

\tableofcontents

\section{Hodge diamonds of Lefschetz fibrations via Macaulay2}

The package \emph{ProjAdjoint} provides Macaulay2 algorithms for defining compactifications of adjoint orbits and of the fibres of symplectic Lefschetz fibrations (SLF) on them. 
The Macaulay2 function can then be used to calculate the corresponding Hodge diamonds. 
An SLF is a fibration $f \colon X \to \mathbb C$ that has only Morse type singularities such that the fibres of $f$ are symplectic submanifolds of $X$ outside the critical set, see \cite{Se}. 
We need the topological information provided by Hodge diamonds to study categories of Lagrangian vanishing cycles on symplectic Lefschetz fibrations. 
These play an essential role in the Homological Mirror Symmetry conjecture \cite{Ko}, where such a category appears as the Fukaya category of a Landau-Ginzburg (LG) model. 
A Landau-Ginzburg model is a Kähler manifold X equipped with a holomorphic function $f \colon X \to \mathbb C$ called the superpotential. 
SLFs are nice examples of LG models. 
In fact, a rigorous definition of the Fukaya category is known only for SLFs and not in any greater generality.

A recent theorem of \cite{GGS1} showed the existence of the structure of SLFs on adjoint orbits of semisimple Lie algebras. 
These adjoint orbits are non-compact spaces. 
In fact, they are isomorphic to cotangent bundles of flag varieties \cite{GGS2}. 
Since Macaulay2 calculates Hodge diamonds for compact varieties, one needs to to compactify these orbits for the calculations. 
Expressing the adjoint orbit as an algebraic variety, we homogenise its ideal to obtain a projective variety which serves as our compactification. 
We then obtain topological data for the total space $X$ as well as for the fibres of $f$.

\begin{remark}
  Choosing a compactification is in general a delicate task: a different choice of generators for the defining ideal of the orbit can result in completely different Hodge diamonds of the corresponding compactification. 
  This happens because the homogenisation of an ideal $I$ can change drastically if we vary the choice of generators for $I$. 
  Our package always chooses the ideal  generated by the entries of the minimal polynomial defining the matrices of the orbit.
\end{remark}

\section{Lefschetz fibrations on adjoint orbits}

Let $H_0$ be an element in the Cartan subalgebra of a semisimple Lie algebra $\mathfrak g$, and let $\mathcal O(H_0)$ denote its adjoint orbit.
It is proved in \cite{GGS1} that for each regular element $H \in \mathfrak g$, the function $f_H\colon \mathcal O(H_0)\rightarrow \mathbb C $ given by $f_H(x) = \inner{H,x}$ gives the orbit the structure of a symplectic Lefschetz fibration.

We compactify the orbit by projectivisation; that is, we homogenise the polynomials by adding an extra variable $t$ to obtain a projective variety.

\section{Example: the case of $H_0 = \diag (2,-1,-1) \in \mathfrak{sl}(3,\mathbb C)$ } % (fold)
\label{sec:2m1m1}

In $\mathfrak{sl} (3, \mathbb C)$, consider the orbit $\mathcal O (H_0)$ of
\[
  H_0 = 
  \begin{pmatrix}
    2 & 0 & 0 \\
    0 & -1 & 0 \\
    0 & 0 & -1
  \end{pmatrix}
\]
under the adjoint action.  
We use Macaulay2 to calculate the Hodge diamonds of a compactification of the adjoint orbit $\mathcal O (H_0)$. 
We fix the element
\[
  H = 
  \begin{pmatrix}
    1 & 0 & 0 \\
    0 & -1 & 0 \\
    0 & 0 & 0
  \end{pmatrix}
\]
to define the potential $f_H$.  
The explicit formula for $f_H$ is given below.

The function \emph{compactOrbit} of \emph{ProjAdjoint} takes one input: a list of numbers corresponding to the diagonal entries of  a matrix $H_0 \in \mathfrak{sl} (n, \mathbb C)$.  
It then outputs a projective variety which is a compactification of the adjoint orbit of $H_0$ in $\mathfrak{sl} (n, \mathbb C)$.

The function \emph{compactFibre} takes three inputs: two lists of numbers corresponding to the diagonal entries of matrices $H$ and $H_0$ in $\mathfrak{sl} (n, \mathbb C)$,  and one complex number $\lambda$.  
The matrix $H$ should be regular.  
The output is a projective variety which is a compactification of the  fibre over $\lambda$.
See \cite{github} for the source code and further documentation.

\subsection{Orbit} % (fold)
\label{sec:2m1m1TheOrbit}

We must first define our base field $\mathbb k$ and ring of polynomials $R$.
Note that the prime number $32749$ is the largest prime that Macaulay2 can
work with (\cite{M2book}).
\begin{verbatim}
  i1 : k = ZZ/32749;

  i2 : R = k[x_1, x_2, y_1..y_3, z_1..z_3, t];
\end{verbatim}
The variable $t$ will be used for projectivisation.  A general element
$A \in \mathfrak{sl} \left( 3, \mathbb{C} \right)$ has the form
\begin{verbatim}
  i3 : A = matrix{
             {x_1, y_1, y_2},
             {z_1, x_2, y_3},
             {z_2, z_3, -x_1 - x_2}
       };

               3       3
  o3 : Matrix R  <--- R
\end{verbatim}
and we will also make use of the identity matrix
\begin{verbatim}
  i4 : Id = id_(k^3);

               3       3
  o4 : Matrix k  <--- k
\end{verbatim}
In our example, the adjoint orbit $\mathcal{O} (H_0)$ consists of all the
matrices with the minimal polynomial $(A + \id)(A - 2\id)$, so we are
interested in the variety cut out by the equation \verb+minPoly+:
\begin{verbatim}
  i6 : minPoly = (A + Id)*(A - 2*Id);

               3       3
  o6 : Matrix R  <--- R

  i7 : I = ideal minPoly;

  o7 : Ideal of R
\end{verbatim}
To obtain a projectivisation of $X$, we first homogenise its ideal $I$, then
take the corresponding projective variety.  Saturating the ideal does not
change the projective variety it defines but can make computations faster.
\begin{verbatim}
  i8 : Ihom = saturate homogenize(I,t);

  o8 : Ideal of R

  i9 : Xproj = Proj(R/Ihom);
\end{verbatim}
One checks with the command \verb+dim Xproj+ that $\dim \overline{X} = 4$.
To check that
$\overline{X}$ is non-singular, we use:
\begin{verbatim}
  i10 : codim singularLocus Ihom 

  o10 = 9
\end{verbatim}
Since the codimension of the singularities is $9$ but the dimension of the
ambient projective space is $8$, we deduce that our projective variety
$\overline{X}$ must be non-singular. 

Now we calculate the Hodge diamond of $\overline X$.  The Hodge numbers $h^{i,j}$ for $i+j \le
4$, $i \ge j$, are computed with the command \verb+hh^(i,j) Xproj+.
Since $\overline{X}$ is non-singular, the other entries of the Hodge diamond
are given by the classical symmetries, as shown below. 

\[
  \begin{array}{ccccccccc}
    &&&& 1 &&&& \\
    &&& 0 && 0 &&& \\
    && 0 && 2 && 0 && \\
    & 0 && 0 && 0 && 0 &  \\
    0 && 0 && 3 && 0 && 0 \\
    & 0 && 0 && 0 && 0 & \\
    && 0 && 2 && 0 && \\
    &&& 0 && 0 &&& \\
    &&&& 1 &&&& 
  \end{array}
\]

\subsection{Regular fibre} % (fold)
\label{sec:2m1m1TheRegularFibre}

To define the regular and critical fibres, we also need the potential,
which in our case is given by:
\begin{verbatim}
  i20 : potential = x_1 - x_2;
\end{verbatim}
The critical values of this potential are $\pm 3$ and $0$.
Since all regular fibres of an SLF are isomorphis, it suffices to chose the regular value $1$.
We then define the regular fibre $X_1$  as the
variety in $\mathfrak {sl} (3, \mathbb C) \cong \mathbb C^8$ corresponding to
the ideal $J$:
\begin{verbatim}
  i21 : J = ideal(minPoly) + ideal(potential-1);

  o21 : Ideal of R
\end{verbatim}
We then homogenise $J$ to obtain a projectivisation $\overline{X}_1$ of the
regular fibre $X_1$:
\begin{verbatim}
  i22 : Jhom = saturate homogenize(J,t);

  o22 : Ideal of R

  i23 : X1proj = Proj(R/Jhom);
\end{verbatim}
We check with the command \verb+dim X1proj+ that $\dim \overline{X}_1 = 3$. We use the command
\begin{verbatim}
  i24 : codim singularLocus Jhom 

  o24 = 9
\end{verbatim}
to test for singularities.  Since this codimension is $9$, we see that
$\overline{X}_1$ is indeed non-singular.
Now
we calculate $h^{i,j}$ for $i+j \le 3$ e $i \ge j$ with the command
\verb+hh^(i,j) X1proj+.
Since $\overline{X}_1$ is non-singular, the other entries of the Hodge diamond
are obtained via the classical symmetries.

\[
  \begin{array}{ccccccc}
    &&& 1 &&& \\
    && 0 && 0 && \\
    & 0 && 2 && 0 & \\
    0 && 0 && 0 && 0  \\
    & 0 && 2 && 0 & \\
    && 0 && 0 && \\
    &&& 1 &&& 
  \end{array}
\]

\begin{remark}
We used the same method to calculate the Hodge diamonds for the singular fibre over $0$ 
and obtained the same Hodge diamond as for the regular fibres.
\end{remark}

\begin{remark}
  More details of this example in appear in \cite{brianThesis}.
\end{remark}

\section{Generalisation: $H_0 = \diag (n, -1, \dotsc, -1) \in \mathfrak{sl}(n+1,\mathbb C)$} 
\label[section]{sec:generalisation}

We generalise our example of $\mathfrak{sl} (3, \mathbb C)$ to
$\mathfrak{sl} (n+1, \mathbb C)$.  To obtain the minimal flag, we set
$H_0 = \diag (n, -1, \dotsc, -1)$ and $H = \diag (1,-1,0, \dotsc, 0)$.  Then 
the diffeomorphism type of the adjoint orbit is given by $\mathcal O
(H_0) \simeq
 T^* \mathbb P^{n}$ (see \cite{GGS2}), and $H$ gives the potential $x_1 - x_2$
as before.  If we  compactify this orbit to $\mathbb P^{n} \times \mathbb
(\mathbb P^{n} )^\ast$, then the Hodge classes of the compactification  are
given by $h^{p,p} = n+1 - \abs{n-p}$ and the
remaining  Hodge numbers are $0$.  An application of the Lefschetz hyperplane
theorem determines all but the Hodge numbers of the middle row of the compactification of
the regular fibre.

\section{Computational corollaries and pitfalls}

The following two corollaries follow immediately from observing the Hodge diamonds we obtained.

\begin{corollary} 
  Let $H_0 = \diag(n,-1, \cdots, -1) \in \mathfrak{sl} (n+1, \mathbb C)$ and $H = \diag(1,-1,0, \cdots, 0)$.  Then the orbit compactifies holomorphically and symplectically to a trivial product.
\end{corollary}

\begin{proof}
For the examples we considered here, \cite{GGS2} showed that  $\mathcal O(H_0)$
can be embedded differentiably   into $\mathbb P^n\times {\mathbb P^n}^{\ast}$. 
As an outcome of our computations, we verify  that the compactifications are  
holomorphically  and symplectically isomorphic to   $\mathbb P^n\times {\mathbb P^n}^{\ast}$ as well.
In fact, 
our package produces a compactification of the orbit embedded into $\mathbb P^{n^2-1}$
and the diamond shows that the compactified orbit has the topological type of a $\mathbb P^n$ bundle over $\mathbb P^n$. These combined
imply the bundle is trivial.
\end{proof}

\begin{corollary} 
An extension of the potential $f_H$ to the compactification  $\mathbb P^n\times {\mathbb P^n}^{\ast}$
cannot be of Morse type; that is, it must have degenerate singularities.
\end{corollary}

\begin{proof} Our potential has singularities at $wH_0, w\in \mathcal W$.  Now
observe that the Hodge diamond of our compactified regular fibres have all zeroes in the middle row, 
hence any extension of the  fibration  to the compactification will have no vanishing cycles. 
However, the existence of a Lefschetz fibration  with singularities and without vanishing cycles  
is prohibited by the fundamental theorem of Picard-Lefschetz theory.
\end{proof}

\begin{remark}
  [Computational pitfalls]
  Macaulay2 greatly facilitates calculations of Hodge numbers that are unfeasible by hand. 
  However, the memory requirements rise steeply with the dimension of the variety -- especially for the Hodge classes $h^{p,p}$.
\end{remark}

\end{document}